\documentclass[12pt,a4]{amsart}
\usepackage{amsmath,amssymb,amsthm,comment}
\usepackage{graphicx}
\theoremstyle{plain}
\newtheorem{thm}{Theorem}[section]
\newtheorem{lem}[thm]{Lemma}
\newtheorem{cor}[thm]{Corollary}

\newtheorem{conj}[thm]{Conjecture}

\newcommand{\R}{\mathbb{R}}

\newcommand{\N}{\mathbb{N}}

\newcounter{mycounter}

\usepackage[dvipdfm,
bookmarks=true,bookmarksnumbered=true,
bookmarksopen=true]{hyperref}
\begin{document}

\title{New bounds for equiangular lines and spherical two-distance sets}
\author{Wei-Hsuan Yu}
\address{Department of Mathematics, Michigan State University,
619 Red Cedar Road, East Lansing, MI 48824}
\email{u690604@gmail.com}
\date{}
\maketitle

\begin{abstract}
A set of lines in $\R^n$ is called equiangular if the angle between each pair of lines is the same. 
We derive new upper bounds on the cardinality of equiangular lines. Let us denote the maximum cardinality 
of equiangular lines in $\R^n$ with the common angle $\arccos \alpha$ by $M_{\alpha}(n)$. We prove that $M_{\frac 1 a} (n) \leq  \frac 1 2 ( a^2-2) ( a^2-1)$ for any $n \in \mathbb{N}$ in the interval $a^2 -2 \leq n \leq  3 a^2-16$ and $a \geq 3$. Moreover, we discuss the relation between equiangular lines and spherical two-distance sets and we obtain the new results on the maximum spherical two-distance sets in $\mathbb{R}^n$ up to $n \leq 417$. 

\end{abstract}

\section{Introduction}

A set of lines in $n$ dimensional Euclidean space $\R^n$ is called \emph{equiangular} if
the angle between each pair of lines is a constant $\theta$.
The constant $\theta$ is called \emph{the common angle} of the equiangular lines. Estimating the maximum size of equiangular lines in $\R^n$ is one of the classical problems in discrete geometry. Let us denote the maximum cardinality 
of equiangular lines in $\R^n$ by $M(n)$. The study on $M(n)$  can be traced from Haantjes in 1948 \cite{Haa48}. After around seven decades research, people only know the answer of $M(n)$ up to $n = 43$.  Recent progress on $M(n)$ can be found in \cite{lem73, lin66, barg14} and their references. Lemmens and Seidel \cite{lem73} solved $M(n)$ for most values of  $n$ if $n \leq 23$. Barg and Yu \cite{barg14} used semidefinite programming (SDP) method to extend the results to $24 \leq n \leq 41$ and $n=43$. Therefore, people know the what is the maximum size of equiangular lines in $\R^n$ up to $n \leq 43$, however for $n=14, 16, 17, 18, 19, 20$ and $ 42$ are still open. We summarize the known lower and upper bounds for those open cases in the  Table \ref{table:open}. For $n=14$ and $16$, the upper bounds are improved in \cite{grea14}, proving no $30$ equiangular lines in $\R^{14}$ and no $42$ equiangular lines in $\R^{16}$. For $n=19$ and $20$, the upper bounds are improved in \cite{yu16}, proving no 76 equiangular lines in $\R^{19}$ and no 96 equiangular lines in $\R^{20}$ due to the nonexistence of some strongly regular graphs \cite{aza15,aza16}.  

\begin{table}[htb]
\caption{Open cases of maximumx eauiangular lines in $\mathbb{R}^n$}
\label{table:open}
  \begin{tabular}{|c|c|c|} \hline
    $n$ & lower bound & upper bound \\
    $14$ & 28 & 29 \\
    $16$ & 40 & 41 \\
    $17$ & 48 & 50 \\ 
    $18$ & 48 & 61 \\
    $19$ & 72 & 75 \\
    $20$ & 90 & 95 \\ 
     42  & 276 &288 \\
\hline
  \end{tabular}
\end{table}

For $\alpha \in [0,1)$, 
let us denote the maximum cardinality 
of equiangular lines in $\R^n$ with the common angle $\arccos \alpha$ by $M_{\alpha}(n)$.
Most of the cases, the upper bounds on the size of equiangular lines are obtained from semidefinite programming method by convex optimization toolkits (CVX) for given dimension and angle, therefore we only can obtain the results for finitely many dimensions. However,  our main result Theorem \ref{thm:main} can obtain new upper bounds of equiangular lines for infinitely many dimensions. Theorem \ref{thm:main} is derived from solving relaxation of symbolic semidefinite programming problems. Also, the bounds can be derived from hand calculation without using any convex optimization software in computer.

\begin{thm}\label{thm:main}
Let us choose $a \geq 3$,  for any $n \in \mathbb{N}$ in the interval $a^2 -2 \leq n \leq  3 a^2-16$ and then
\begin{align*}
M_{\frac 1 a} (n) \leq  \frac 1 2 ( a^2-2) ( a^2-1)
\end{align*}
\end{thm}

If the dimension $n= a^2-2$, then the main theorem will obtain upper bounds $n(n+1)/2$ which is  nothing but Gerzon's bound \cite{lem73}. However, the same values will remain the upper bounds for dimension $n$ at least to dimension $n=3a^2-16$. Let us use the Table \ref{table:ex} to demonstrate our results for smaller dimensions. 

\begin{table}[htb]
\caption{Main results for first few cases}
\label{table:ex}
  \begin{tabular}{|c|c|c|} \hline
      angle      &  upper bounds & valid dimensions \\ \hline
    1/3  &  28 &7 -- 11 \\  \hline
    1/5  & 276 & 23 -- 59 \\ \hline
    1/7 & 1128 & 47 -- 131 \\  \hline
    1/9& 3160 &79 -- 227 \\ \hline
    1 /11 & 7140 &119 -- 347 \\ \hline
    1 /13 &  14028 &167 --  491  \\  \hline
  \end{tabular}
\end{table}

We observe this pattern form the Table 3 in Barg-Yu \cite{barg14} and eventually we prove that the pattern is true in general to infinitely many dimensions $n$. Moreover, the table in King-Tang \cite{kin16} experimentally verify our main results for $n \leq 400$. 

A set of unit vectors $S = \{x_1, x_2, . . . \} \subset \R^n$ is called a spherical two-distance set if  $\langle x_i
, x_j \rangle \in
\{a, b\}$ for some $a, b$ and all $i \neq j$. To determine maximum size of a spherical two-distance set in $\R^n$ is a classical problem in discrete geometry. Recent progress on this topic can be found in \cite{mus09, barg13}. Currently, we know the maximum size of spherical two-distance sets in $\R^n$ for $n \leq 93$ except for $n=46$ and $78$ \cite{barg13}. We extend the results up to $n=417$

\begin{thm}\label{thm:main1}
Maximum size of a spherical two-distance set in $\R^n$ is $ \frac{n(n+1)}{2}$ for $7 \leq n \leq 417$ except for $n=22, 46, 78, 118, 166, 222, 286$ and  $358$ which are all square of odd integers minus three, i.e. $n = (2k+1)^2-3$ for $k=2, 3, \cdots, 9$.
\end{thm}

In section 2, we will discuss some known results for equiangular lines and review the semidefinite programming method on equiangular lines. We take the relaxation of matrix inequality and solve the semidefinite programming problems to prove Theorem \ref{thm:main}. Section 3, we will review historic results of spherical two-distance sets. We use the bounds for equiangular lines in $\R^{n+1}$ to offer upper bounds for spherical two-distance sets in $\R^n$ and then prove Theorem \ref{thm:main1}. In section 4, we will have some discussions, remarks and conjectures.   

\section{New bounds for equiangular lines in $\R^n$}
We review some known results of equiangular lines here. 

\begin{lem} \cite[Theorem 3.4]{lem73}
Neumann proved that 
\[
M_\alpha(n) \leq 2n \quad \text{if } 1/\alpha \text{ is not an odd integer}.
\]
\end{lem}
Therefore, we are more interested in the cases where $1/\alpha$ is an odd integer.

\begin{lem}\cite[Theorem 3.5 (Gerzon)]{lem73}
If we have $M$ equiangular lines in $\R^n$, then 
$$M \leq \frac{n(n+1)} 2.$$ 
Moreover, if equality holds, then the common angle  $\theta= \cos^{-1} \sqrt{ \frac1{n+2}}$ and $n=2, 3$ or $(2k+1)^2-2$ for $k \in \mathbb{N}$.
\end{lem}
Surprisingly, we only know four examples $n=2, 3, 7$ and $ 23$ to attain Gerzon bounds. Also, for next two candidates $n=47$ and $79$ are impossible to attain Gerzon bounds due to nonexistence of tight spherical $5$-designs in $\R^{47}$ and $\R^{79}$\cite{ban04}. The link between equiangular lines and tight spherical 5-designs is discussed in \cite{del77} and \cite[Theorem 4.2]{yu16}.  

Lemmens and Seidel \cite[Theorem 3.6]{lem73} showed that
  \begin{equation}\label{eqn:rel}
    M_\alpha(n) \le \frac {n(1-\alpha^2)}{1-n\alpha^2} \quad \text{ in the cases where } 1-n\alpha^2 > 0.
  \end{equation}
This inequality is called the Lemmens--Seidel relative bound as opposed to the Gerzon absolute bound. Okuda and Yu \cite{oy16} derived new relative bounds for equiangular lines. Consequently, that result proved the nonexistence of tight harmonic index 4-designs. For more details of harmonic index $t$-designs, one can check the references \cite{ban13}, \cite{ban16}. 

Barg and Yu \cite{barg14} used SDP method to obtain better upper bounds than Gerzon bounds up to $n \leq 136$. King and Tang \cite{kin16} used classical pillar method in conjunction with SDP methods on spherical two-distance sets and equiangular lines in $\R^n$ to improved the upper bounds up to $n\leq 400$. If we fixed the angle $\theta$, Bukh \cite{buk16} proved that the size of equiangular lines is at most linear in the dimension. Later, Balla, Dr\"{a}xler, Keevash and Sudakov \cite{bal16} proved that there are at most $2n-2$ equiangular lines in $\R^n$ for sufficiently large $n$ and fixed angle $\theta$.  

The motivation of this paper is that we observe the pattern of SDP bounds on the size of equiangular lines in Table 3 in  \cite{barg14}. The values of upper bounds 276, 1128, 3160 and 7140 strikingly show up so many times. We find that those numbers satisfy the formula $\frac {n(n+1)}2$ for $n=23, 47, 79$ and $119$ respectively and the values of $n$ also satisfy the pattern $(2k+1)^2-2$ for consecutive positive integers $k= 2, 3, 4$ and $5$. We wonder that this pattern should be true even when $n$ goes to infinity. Eventually, we prove it by solving relaxation of symbolic semidefinite programming problems. Our techniques are based on the semidefinite programming (SDP) methods for codes
on the unit sphere introduced by Bachoc and Vallentin \cite{bac08a}.
For most of cases to use SDP method, it is necessary to use optimization software.
It should be emphasized that our theorem provides upper bounds for $M_{1/a}(n)$ for arbitrarily large odd integers $a$
and the proof can be followed by hand calculations without using any convex optimization software.

We define a family of polynomial functions called Gegenbauer polynomials $P_k^n$ for $k = 0, 1, 2,\dots$ .
Let $P_0^n(u)= 1$, $P_1^n(u) =  u$ and for $k \geq 2$, 
$$P_k^n(u) = \frac{(2k+n-4)u P^n_{k-1}(u) -(k-1)P^n_{k-2}(u)}{k+n-3}$$

We use symbols $S^n_k(u,v,t)$ and $Y^n_k(u,v,t)$ which are the same in \cite{bac08a}. Define an infinite-size matrix-valued function $Y^n_k$ by
\[
(Y^n_k)_{i,j}(u,v,t) = \lambda_{i,j} P_i^{n+2k}(u) P_j^{n+2k}(v) Q_k^{n-1}(u,v,t),
\]
where 
\begin{align*}
Q_k^{n-1}(u,v,t) &= (1-u^2)^{k/2} (1-v^2)^{k/2} P_k^{n-1}(\frac{t-uv}{\sqrt{(1-u^2)(1-v^2)}}), \\
\lambda_{i,j} &= \frac{n+2k}{n} (h^{n+2k}_i h^{n+2k}_j)^{1/2}
\end{align*}
with $h^{n+2k}_i = \binom{n+2k+i-1}{n+2k-1} - \binom{n+2k+i-3}{n+2k-1}$

We define
\[
S^n_k(u,v,t) = \frac{1}{6} \sum_{\sigma} Y_k^n(\sigma(u,v,t)),
\]
where the sum is over all permutations on 3 elements.

We also define
\begin{align*}
W(x)&:= \begin{pmatrix}1&0\\0&0\end{pmatrix} +
\begin{pmatrix}0&1\\1&1\end{pmatrix} (x_1+x_2)/3 +
\begin{pmatrix}
0&0\\0&1
\end{pmatrix} (x_3+x_4+x_5+x_6), \\
S^{n}_k(x;\alpha,\beta) &:= S^{n}_k(1,1,1)+S^{n}_k(\alpha,\alpha,1)x_1 + S^{n}_k(\beta,\beta,1) x_2 + S^{n}_k(\alpha,\alpha,\alpha) x_3 \\
& \quad \quad + S^{n}_k(\alpha,\alpha,\beta) x_4 + S^{n}_k(\alpha,\beta,\beta) x_5 + S^{n}_k(\beta,\beta,\beta) x_6
\end{align*}
for each $x = (x_1,x_2,x_3,x_4,x_5,x_6) \in \R^6$ and $\alpha, \beta \in [-1,1)$.

We follow the SDP method to obtain upper bounds on the size of  equiangular lines in $\R^n$.

\begin{thm}[\cite{bac08a},{\cite[Theorem 2.1]{barg14}}]\label{fact:SDP-problem}
\[
M_a(n) \leq \max \{ 1 + (x_1 + x_2)/3 \mid x = (x_1,\dots,x_6) \in \Omega^{n}_{a,-a} \}
\]
where the subset $\Omega^{n}_{a,-a}$ of $\R^{6}$ is defined by
\[
\Omega_{a,-a}^{n} := \{\, x = (x_1,\dots,x_6) \in \R^{6} \mid \text{ $x$ satisfies the following four conditions } \,\}.
\]
\begin{enumerate}
\item $x_i \geq 0$ for each $i = 1,\dots,6$.
\item $W(x)$ is positive semidefinite.
\item $3 + P^{n}_k(a) x_1 + P^{n}_k (-a) x_2 \geq 0$ for each $k=1,2,\dots.$
\item $S^{n}_k(x;a,-a)$ is positive semidefinite for each $k = 0,1,2,\dots.$
\end{enumerate}
\end{thm}

Surprisingly, we find that only $S^n_3(x;a,-a)$ and $S^n_1(x;a,-a)$ are crucial to prove Theorem \ref{thm:main}. Also, we know that if a matrix is positive semidefinite, then all the diagonal entries of this matrix are positive. We take the relaxation of the constraints in Theorem \ref{fact:SDP-problem} as follows :  $W(X)$ is positive semidefinite, $(S^n_3)_{1,1} \geq 0$ and $(S^n_1)_{1,1} \geq 0$, where the lower index means the $(1,1)$ entry of that matrix. Therefore, we have the following corollary:

\begin{cor}\label{r}
$M_a(n)$ is bounded above by the solution of following semidefinite programming problem. 

$$\max  : 1+(x_1+x_2)/3$$
subject to 
$$\text{det}(W) \geq 0$$
$$(S^{n}_1(x;a,-a))_{1,1} \geq 0 $$
$$(S^{n}_3(x;a,-a))_{1,1} \geq 0 $$
\end{cor}

We list explicit formula for $(S^{n}_3)_{1,1}$, obtained by direct computations. 
\begin{lem}\label{lem:S3explicite}
For each $-1 < \alpha < 1$,
\begin{align*}
(S^{n}_3)_{1,1}(1,1,1) &= 0, \\
(S^{n}_3)_{1,1}(\alpha,\alpha,1) &= \frac{n(n+2)(n+4)(n+6)}{3(n-1)(n+1)(n+3)}
\alpha^2 (1-\alpha^2)^3, \\
(S^{n}_3)_{1,1}(\alpha,\alpha,\alpha) &=
-\frac{n(n+2)(n+4)(n+6)}{(n-2)(n-1)(n+1)(n+3)}
(\alpha-1)^3 \alpha^3 ((n-2)\alpha^2-6\alpha-3), \\
(S^{n}_3)_{1,1}(\alpha,\alpha,-\alpha) &=
-\frac{n(n+2)(n+4)(n+6)}{(n-2)(n-1)(n+1)(n+3)}
\alpha^3 (\alpha+1)^3 ((n-2)\alpha^2 +6\alpha-3).
\end{align*}
\end{lem}

\begin{thm}\label{rr}
 $$M_a(n) \leq  \frac 1 2 ( \frac 1 {a^2}-2) ( \frac 1 {a^2}-1),$$ 

for any $n \in \mathbb{N}$ in the interval $\frac1{a^2} -2 \leq n \leq \frac 3 {a^2}-16$ and $a \leq \frac 13$.
\end{thm}

It is not hard to see that Theorem \ref{rr} is equivalent to Theorem \ref{thm:main}.

\begin{proof}
We know that $M_a(n)$ is an increasing function for $n$, i.e. if $n_1 \leq n_2$, then $M_a(n_1) \leq M_a(n_2)$.  If we can prove that  $M_a(\frac{3}{a^2}-16 ) \leq  \frac 1 2 ( \frac 1 {a^2}-2) ( \frac 1 {a^2}-1)$, then the statement of Theorem \ref{rr} is true. Therefore, we replace $n=\frac{3}{a^2}-16$ to the formula of $S^n_k$. By Corollary \ref{r}, we know that the solution of following optimization problem offers upper bounds for $M_a(n)$.

For simplicity we put $A = (x_1+x_2)/3$, $B = x_3+x_5$ and $C = x_4+x_6$. 
$$\max  (1+A)$$

subject to 
\begin{align}
\label{eq:1}
&A+ \frac{2a^4(3a+1)}{(6a^2-1)(a+1)^3} B +  \frac{2a^4(3a-1)}{(6a^2-1)(a-1)^3}C \geq 0 \\ 
\label{eq:2}
& A+ \frac a {1+a} B + \frac a {a-1} C \geq 0 \\
\label{eq:3}
& A(A-1) \leq B+C
\end{align}

The condition \eqref{eq:1}, \eqref{eq:2}, and \eqref{eq:3} are exactly $(S^n_3)_{1,1} \geq 0, (S^n_1)_{1,1} \geq 0$ and $ \text {det } W \geq 0$, respectively. We choose suitable $t$, where $t=\frac{-16a^6}{(6a^2-1)(a+1)^2(a-1)^2}$ such that 
$$
t \frac a {1+a} +  \frac{2a^4(3a+1)}{(6a^2-1)(a+1)^3} = t \frac a {a-1} +  \frac{2a^4(3a-1)}{(6a^2-1)(a-1)^3} =  \frac{-2a^4(5a^2-1)}{(6a^2-1)(a-1)^2(a+1)^2}.
$$
The motivation to choose this $t$ is that we want to use \eqref{eq:3} which requires $B$ and $C$ having the same coefficients. 
If we consider that $t$ times \eqref{eq:2} and plus \eqref{eq:1}, i.e. $t$\eqref{eq:2} + \eqref{eq:1}, we will get
\begin{align*}
&(t+1) A +( t \frac a {1+a} +  \frac{2a^4(3a+1)}{(6a^2-1)(a+1)^3}) B+(t \frac a {a-1} +  \frac{2a^4(3a-1)}{(6a^2-1)(a-1)^3}) C \geq 0 \\
 \Rightarrow &(t+1) A +( t \frac a {1+a} +  \frac{2a^4(3a+1)}{(6a^2-1)(a+1)^3}) (B+C) \geq 0 \\
 \Rightarrow & -\frac{10a^6+13a^4-8a^2+1}{(6a^2-1)(a-1)^2(a+1)^2} A  -  \frac{2a^4(5a^2-1)}{(6a^2-1)(a-1)^2(a+1)^2}(B+C) \geq 0 \\
& \text{Notice that $ \frac{a^4(5a^2-1)}{(6a^2-1)(a-1)^2(a+1)^2} \geq 0$ if $a \leq 1/3$.} \\
& \text{Then, we use \eqref{eq:3} to replace $B+C$ to $A(A-1)$.} \\
 \Rightarrow& -\frac{10a^6+13a^4-8a^2+1}{(6a^2-1)(a-1)^2(a+1)^2} A  -  \frac{2a^4(5a^2-1)}{(6a^2-1)(a-1)^2(a+1)^2}A(A-1) \geq 0 \\
 \Rightarrow& -\frac{10a^6+13a^4-8a^2+1}{(6a^2-1)(a-1)^2(a+1)^2}  \geq   \frac{2a^4(5a^2-1)}{(6a^2-1)(a-1)^2(a+1)^2}(A-1) \\
\Rightarrow &  -\frac{10a^6+13a^4-8a^2+1}{2a^4(5a^2-1)}=\frac{1-3a^2-2a^4}{2a^4} \geq A-1  \\  
\Rightarrow&  A \leq  \frac{1-3a^2}{2a^4}
\end{align*}
Then, 
\begin{align*}
A+1 &\leq  \frac{1-3a^2}{2a^4} +1 = \frac{1-3a^2+2a^4}{2a^4}= \frac{(1-2a^2)(1-a^2)}{2a^4} \\
&=\frac 1 2 ( \frac 1 {a^2}-2) ( \frac 1 {a^2}-1). 
\end{align*}
\end{proof}

\section{new bounds for spherical two-distance sets}
A set of unit vectors $S = \{x_1, x_2, . . . \} \subset \R^n$ is called a spherical two-distance set if  $\langle x_i
, x_j \rangle \in
\{a, b\}$ for some $a \neq b$ and for all $i \neq j$. The study of spherical two-distance set can be traced from Delsarte, Goethals and Seidel in 1977 \cite{del77}. Estimating the maximum size $g(n)$ of
such a set is a classical problem in distance geometry that has
been studied for several decades. Equiangular lines also can be regarded as a special type ($b=-a$) of a spherical two-distance set. 

We begin with an overview of known results. A lower
bound on $g(n)$ is obtained as follows. Let
$e_1,\dots,e_{n+1}$ be the standard basis in $\R^{n+1}$.
The points $e_i+e_j, i\ne j$ form a spherical two-distance set in the
plane $x_1+\dots+x_{n+1}=2$ (after scaling), and therefore
   \begin{equation}\label{eq:n}
     g(n)\ge n(n+1)/2, \quad n\ge 2.
  \end{equation}

The first major result for upper bounds was obtained by
Delsarte, Goethals, and Seidel \cite{del77}. They proved that, irrespective of the actual
values of the distances, the following harmonic bound holds true:
  \begin{equation}\label{eq:dgs}
   g(n)\leq n(n+3)/2.
  \end{equation}
They also showed that this bound is tight for dimensions $n=2,6,22$ in which cases it is related to sets of equiangular lines in dimension $n+1.$ Moreover, if a spherical 2-distance set attains above bound, then it forms a tight spherical
4-design \cite{del77}. The results of tight spherical 4-designs by Bannai et al.~\cite{ban04}, and Nebe and Venkov \cite{neb12} imply that $g(n)$ can attain the harmonic bound
only if $n=(2k+1)^2-3,  k \ge 1$ with the exception of an infinite sequence of values of
$k$ that begins with $k=3,4,6,10,12,22,38,30,34,42,46$.

 Musin used the linear programming method to obtain maximum size of spherical two-distance sets in $\R^n$ for $7 \leq n \leq 39$ except $n=23$ \cite{mus09}. 
\begin{thm}\cite{mus09} (Musin)
 $$g(n)=n(n+1)/2  
\text{, if }
7\le n\leq 39, n\ne 22,23.$$
Moreover, $g(23)=276$ or $277$.
\end{thm}

Barg and Yu used SDP mothods to extend the results up to $n \leq 93$. In particular, they proved $g(23)=276$. We summerize the results as follows.
\begin{thm}\cite{barg13} (Barg-Yu)\\
We have $g(2)=5, g(3)=6, g(4)=10, g(5)=16, g(6)=27,g(22)=275,$
   \begin{align}
     &g(n)=n(n+1)/2, \quad 7\le n\le 93,
n\ne 22,46,78. \label{eq:results1}
   \end{align}
\end{thm}
The exact answers for $g(n)$ remain open for $n=46,78$ and $n \geq 94$. Strikingly, in this paper we extend the results up to $n=417$. We know most values of $g(n)$ if $n \leq 417$.

\begin{thm} 
$$g(n)=\frac{n(n+1)}{2}$$ for $7 \leq n \leq 417$, except $n =22, 46, 78, 118, 166, 222, 286, 358$ which are all square of odd integers minus three, i.e. $n = (2k+1)^2-3$ for $k=2, 3, \cdots, 9$.
\end{thm}
\begin{proof}

To proof Theorem \ref{thm:main1}, we require two lemmas as follows. 
\begin{lem}\cite{mus09}(Musin) \label{lem1}
If $S$ is a spherical two-distance set with inner product values $a$ and $b$, and $a+b \geq 0$, then 

$$
|S| \leq \frac{n(n+1)}2
$$
\end{lem} 
Therefore, to obtain exact answer of $g(n)$, what we need to worry is that $g(n) > \frac {n(n+1)}2$ if $a+b <0$.  In \cite{barg13}, Barg and Yu used SDP method and nontrivial convex optimization toolkit (SOSTOOL) to obtain the rigorous upper bounds for $a+b < 0$ case. However, we offer another point of view to deal with it. We thank Alexey Glazyrin who suggested us to use bounds for equiangular sets in order to get bounds for two-distance sets.
In short, the maximum size of equiangular lines in $\R^{n+1}$, $M(n+1)$ offers an upper bound for a spherical two-distance set in $\R^n$ when $a+b<0$.  

\begin{lem}\cite{del77} \label{lem2}
If $S$ is a spherical two-distance set in $\R^n$ with $a+b<0$, then it leads to equiangular lines with size $|S|$ in $R^{n+1}$.
\end{lem}
\begin{proof}
Let $S = \{ x_1,x_2, \cdots \}$ with $\langle x_i, x_j \rangle= a $,  or $b$  if $ i \neq j$, and $a+b< 0$. Also, all of the $x_i$ are unit vectors in $\R^n$. We can define constant $R$ and $\theta$ such that 
$$ 1-a = R^2(1-\cos \theta) \quad \text{  and  } \quad 1-b=R^2(1+ \cos \theta),$$
where $R>1$, since $a+b <0$. 

Then, we define $Y=\{y_1,y_2, \cdots \}$ and $y_i = (\frac{x_i} R,\frac {\sqrt{R^2-1}} R) \in \R^{n+1}$. It is not hard to see that  $|Y|=|S|, \langle y_i, y_i \rangle= \frac 1 {R^2} + \frac {R^2-1} {R^2} =1$ and $\langle y_i, y_j\rangle = \frac {\langle x_i, x_j \rangle}{R^2} + \frac {R^2-1} {R^2}  = \pm \cos \theta$. Therefore, Y leads to equiangular lines in $\R^{n+1}$.
\end{proof}

By Lemma \ref{lem1} and Lemma \ref{lem2}, we have 
$$g(n) \leq \max\{ M(n+1), \frac {n(n+1)} 2 \}$$

One can check the table in \cite{kin16} for the upper bounds of equiangular lines in $\R^n$ for $ 44 \leq n \leq 400$ in conjunction with the table in \cite{barg14} for $n \leq 139$. Then, we can find that $M(n+1) \leq \frac{n(n+1)}2 $ for $ 7 \leq n \leq 400$, except $n = 22, 46, 78, 118, 166, 222, 286, 358$ which are all square of odd integers minus three.
Furthermore, if we follow Theorem 5 and 6 in \cite{kin16} and use CVX to calculate SDP bounds for equiangular lines and required spherical two-distance sets,  we can extend the results up to $n=417$. The reason to stop at $n=417$ is that the SDP bound for $M_{\frac 1 9}(419) = 88808$ which is greater than $\frac {419*418}2 = 87571$. So, we have nice upper bound for equiangular lines up to $n=418$ and then we can obtain the result for $g(n)$ up to $n=417$. We list the experimental results for $n=401$ to $419$ in Table \ref{table2} which is not listed in \cite{kin16}.

\begin{table}[htb]
\tiny
\caption{Upper bounds for equiangular lines in $\R^n$ for $401 \leq n \leq 419$ }
\label{table2}
  \begin{tabular}{|c|c|c|c|c|c|c|c|c|c|c|c|c|c|c|} \hline
      n      &  $\frac 1 5$ & $\frac 1 7$ & $\frac 1 9 $ & $\frac 1 {11}$ &$\frac 1 {13}$ &$\frac 1 {15}$& $\frac 1 {17}$ &$\frac 1 {19}$
&$\frac 1 {21}$ &$\frac 1 {23}$ & $\frac 1 {25}$& $\frac 1 {27}$&  max & $\frac{n(n+1)}2$ \\  \hline
 
401&   17734&   40215&   57440&   22984&   14028&   24976&   41328&   64620&    4411&  1654.1	&1117.1&  890.02&   64620&   80601	\\	
402&   17874&   40366&   58634&   23836&   14028&   24976&   41328&   64620&  4535.4&  1671.3&1124.9&  894.97&   64620&   81003	\\	
403&   18015&   40517&   59872&   24749&   14028&   24976&   41328&   64620&  4666.3&  1688.8&1132.8&  899.95&   64620&   81406	\\	
404&   18158&   40668&   61158&   25730&   14028&   24976&   41328&   64620&  4804.3&  1706.5&1140.7&  904.96&   64620&   81810	\\	
405&   18303&   40820&   62495&   26786&   14028&   24976&   41328&   64620&    4950&  1724.5	&1148.7&       910&   64620&   82215	\\	
406&   18449&   40972&   63885&   27667&   14028&   24976&   41328&   64620&    5104&  1742.8	&1156.8&  915.07&   64620&   82621	\\	
407&   18597&   41124&   65332&   27825&   14028&   24976&   41328&   64620&  5267.1&  1761.4	&1165&   920.17&   65332&   83028	\\	
408&   18746&   41277&   66839&   27983&   14028&   24976&   41328&   64620&    5440&  1780.4	&1173.2&  925.31&   66839&   83436	\\	
409&   18898&   41430&   68411&   28143&   14028&   24976&   41328&   64620&  5623.7&  1799.6&1181.6&  930.47&   68411&   83845	\\	
410&   19050&   41584&   70051&   28305&   14028&   24976&   41328&   64620&  5819.4&  1819.2&1190&  935.67&   70051&   84255	\\	
411&   19205&   41738&   71764&   28467&   14028&   24976&   41328&   64620&    6028&  1839.1	&1198.4&  940.91&   71764&   84666	\\	
412&   19361&   41892&   73555&   28631&   14028&   24976&   41328&   64620&    6251&  1859.3	&1207&  946.17&   73555&   85078	\\	
413&   19520&   42047&   75429&   28796&   14028&   24976&   41328&   64620&    6490&  1879.9	&1215.6&  951.47&   75429&   85491	\\	
414&   19680&   42202&   77393&   28963&   14028&   24976&   41328&   64620&  6746.7&  1900.8&1224.3&   956.8&   77393&   85905	\\	
415&   19841&   42358&   79453&   29130&   14028&   24976&   41328&   64620&  7023.1&  1922.1&1233.1&  962.17&   79453&   86320	\\	
416&   20005&   42514&   81616&   29300&   14028&   24976&   41328&   64620&  7321.6&  1943.8&1242&  967.57&   81616&   86736	\\	
417&   20171&   42670&   83890&   29470&   14028&   24976&   41328&   64620&    7645&  1965.9	&1251&973&   83890&   87153		\\
418&   20338&   42827&   86284&   29642&   14028&   24976&   41328&   64620&  7996.5&  1988.3&1260.1&  978.47&   86284&   87571	\\	
419&   20508&   42984&   88808&   29815&   14028&   24976&   41328&   64620&    8380&  2011.2	&1269.2&  983.97&   88808&   87990	\\  \hline	
  \end{tabular}
\end{table}

\end{proof}

\section{discussion and future work}
As we can see the pattern of maximum spherical two-distance set, most of the cases, $g(n) = \frac {n(n+1)}2$. Therefore, we have a conjecture for $g(n)$. 
\begin{conj} \label{conj1}
$$g(n) = \frac {n(n+1)}2 \quad \text{if}\quad  n \neq (2k+1)^2-3 \quad \text{for some  } k \in \mathbb{N}$$ 
\end{conj}
At least this conjecture is true for $n \leq 417$ since Theorem \ref{thm:main1} holds.
We also have some clues to prove conjecture \ref{conj1}.  We check the table of upper bounds of equiangular lines in \cite{kin16} and see that in general case, the upper bound is $\frac {n(n+1)}2$ for $n=(2k+1)^2-2$ for some positive integer $k$.  

\begin{thm}\cite{kin16}
For $ 44 \leq n \leq 400$, the upper bound of maximum number of equiangular lines in $\R^n$ is 

\begin{align*}
 M(n) \leq 
\begin{cases}
\frac{4n(k+1)(k+2)}{(2k+3)^2-n}, & \text{$n= 44,45,46,76,77,78,117,118,166,222,286,358.$ (case A)}\\ \\
\frac {n_k(n_k+1)}2,  & \text{$n_k= (2k+1)^2-2$, other $n$ between $44$ and $400$. (case B)}
\end{cases}
\end{align*}
where $k$ is the largest positive integer such that $(2k+1)^2 - 2 \leq n$.
\end{thm}
Notice that for the case A, the bounds are attained by the relative bounds \eqref{eqn:rel} in $\R^n$ for the angle $\cos^{-1} \frac 1 {2k+3}$. For case B, the bounds are attained for angle $\cos^{-1}\frac1{2k+1}$. We notice that the relative bound is increasing by dimension $n$ when angle is fixed.

\begin{lem}\label{lem:inc}
If $f_a(n) = \frac {n(1-a^2)}{1-na^2}$ is the relative bound for equiangular lines in $\R^n$, with common angle $\cos^{-1}a$, then $f$ is an increasing function for dimension $n$  when the angle is fixed and also an increasing function for $a$, when the dimension $n$ is fixed.  
\end{lem}
\begin{proof}
$\frac {df(n)}{d n}=\frac{(1-a^2)(1+na^2)}{(1-na^2)^2} \geq 0$ for all $a<1$ and $n \in \N$.  \\
$\frac {df(n)}{d a}=\frac{2an(n-1)}{(1-na^2)^2} \geq 0$ for all $a<1$ and $n \in \N$.  
\end{proof}

When dimension $n$ is given, there exists an unique $k \in \N $ such that $(2k+1)^2-2 \leq n < (2k+3)^2-2$. By Lemma \ref{lem:inc}, we check the case $n=(2k+3)^2-3$ which is the largest dimension for $(2k+1)^2-2 \leq n < (2k+3)^2-2$. Then, its upper bound dominates all the bounds in case A for fixed $k$. Therefore,   

\begin{align*}
M(n)  \leq 
\begin{cases}
\frac{4n(k+1)(k+2)}{(2k+3)^2-n}, & \text{$n=(2k+3)^2-3$ (case A)}\\ \\
\frac {n_k(n_k+1)}2,  & \text{$n_k= (2k+1)^2-2$  (case B)}
\end{cases}
\end{align*}
The case A bound is $\frac 8 3 (2k^2+6k+3)(k+1)(k+2)$ which is smaller than case B upper bound $(4k^2+4k-1)(2k^2+2k)$ when $k \geq 9$ i.e. $n=438$. Therefore, the case A will not offer the upper bounds for $M(n)$ when $n$ is bigger enough.    
\begin{conj}\label{conj2}
If $n \geq 358$, then 
$$M(n) =  M_{\frac 1 {2k+1}}(n)  \leq \frac{((2k+1)^2-2)((2k+1)^2-1)}2,
$$ 
where $k$ is the unique positive integer such that $(2k+1)^2-2 \leq n < (2k+3)^2-2$.
That is said starting from $n_k=(2k+1)^2-2$ for some $ k \in \N$, there will be a long range of dimensions $n$ having the same upper bounds for equiangular lines for any angle.  
\end{conj}
The motivation of this conjecture is that we believe that $M_{\alpha}(n)$  is an unimodal distribution and the single highest value occurs at the angle $\cos^{-1} \frac 1 {2k+1}$ if $(2k+1)^2-2 \leq n <  (2k+3)^2-2$ for $k \in \N$. Notice that if the Conjecture \ref{conj2} is true, then Conjecture \ref{conj1} follows. By Theorem \ref{thm:main}, we know that $M_{\frac 1 {2k+1}}(n)  \leq \frac{((2k+1)^2-2)((2k+1)^2-1)}2$ since $3(2k+1)^2-16 > (2k+3)^2-2$ if $k >1$. In conjunction with Lemma \ref{lem:inc} and above discussion, we can obtain that $M_{\alpha}(n)  \leq \frac{((2k+1)^2-2)((2k+1)^2-1)}2$ if $\alpha \leq \frac 1 {2k+1}$. Therefore, we only need to take care of $M_{\alpha}(n)$ if $\alpha > \frac 1 {2k+1}$. Namely, we need to find upper bounds for $M_{\frac 1 {2k-1}}(n), \cdots, M_{\frac 1 {5}}(n)   ,  M_{\frac 1 {3}}(n)$.  Currently, we only know  $M_{\frac 1 {3}}(n) \leq 2n-2$ in general. For other angle, it is still open problem. However, for smaller dimensions $n \leq 400$, $M_{\frac 1 {5}}(n)$ and  $M_{\frac 1 {7}}(n)$ have been estimated in \cite{kin16}. For larger $n$, the estimates on $M_{\frac 1 {5}}(n)$ require the upper bounds of spherical two-distance sets in $\R^n$ with inner product $1/13$ and $-5/13$. The best known bounds are obtained by SDP method in convex optimization software for given $n$. Therefore, when $n$ is big, we have no clues. However, in \cite{bal16}, they proved that for fixed angle the size of equiangular lines in $\R^n$ is at most $2n-2$ if $n$ is large enough. It looks like close to prove this conjecture. However, the bounds in \cite{bal16} required the angle fixed. If n increases then so does k in Conjecture  \ref{conj2}. 

It is well-known that the existence of tight spherical 5-designs in $\R^n$ is equivalent to the existence of equiangular lines attaining Gerzon bound \cite{del77, yu16}. Therefore, we like to emphasize that if we can have any new examples of tight spherical 5-designs, then a corollary of our Theorem \ref{thm:main} will obtain $M(n)$ for large values of $n$. For instance, if $n=119$, there exists a tight spherical 5-design in $\R^n$, then it will give arise to $7140$ equiangular lines in $\R^n$. In conjunction with our Theorem \ref{thm:main}, we will know that $M(n)=7140$ for $ 119 \leq n \leq 347$. More than two hundred dimensions of maximum equiangular lines problems will be solved. This phenomenon stays true for all $n= (2k+1)^2-2$ for some $k \in \N$.

\section*{Acknowledgments.}
We would like to thank Alexey Glazyrin and Alexander Barg for useful comments. 

\providecommand{\bysame}{\leavevmode\hbox
to3em{\hrulefill}\thinspace} \providecommand{\href}[2]{#2}
\bibliographystyle{amsalpha}

\begin{thebibliography}{A}

\bibitem{aza15}
J. Azarija and T. Marc, \emph{There is no $(75,32,10,16)$ strongly regular graph}, preprint,
\href{http://arxiv.org/abs/1509.05933}{arXiv:1509.05933}.

\bibitem{aza16}
J. Azarija and T. Marc, \emph{There is no $(95, 40, 12, 20)$ strongly regular graph}, preprint,
\href{http://arxiv.org/abs/1603.02032}{arXiv:1603.02032}. 


\bibitem{bac08a}
C.~Bachoc and F.~Vallentin, \emph{New upper bounds for kissing
numbers from
  semidefinite programming}, J. Amer. Math. Soc. \textbf{21} (2008), \href{http://www.ams.org/journals/jams/2008-21-03/S0894-0347-07-00589-9/home.html}{909--924}.



\bibitem{bal16}
I. Balla, F. Dr\"{a}xler, P. Keevash, and B. Sudakov,
\emph{Equiangular lines and spherical codes in Euclidean space}
available at \href{http://arxiv.org/abs/1606.06620}{arXiv:1606:06620}.


\bibitem{ban04}
E. Bannai, A. Munemasa, and B. Venkov, \emph{The nonexistence of certain
tight spherical designs,} St. Petersburg Math. J. \textbf{16} (2005), 609-625.


\bibitem{barg13}
A. Barg and W.-H. Yu, \emph{New bounds for spherical two-distance
set,} Experimental Mathematics, \textbf{22} (2013),
\href{http://www.tandfonline.com/doi/abs/10.1080/10586458.2013.767725#.VCJb9_l_txE}{187--194}.

\bibitem{barg14}
A. Barg and W.-H. Yu, \emph{New bounds for equiangular lines}, Discrete Geometry and Algebraic Combinatorics, A. Barg and O. Musin, Editors, AMS Series: Contemporary Mathematics, \href{http://www.ams.org/books/conm/625/}{vol. 625}, 2014, pp.111--121. 

\bibitem{ban13}
E. Bannai, T. Okuda, and M. Tagami, \emph{Spherical designs of harmonic index $t$,} J.~Approx.~Theory, Volume 195, July 2015,  \href{http://www.sciencedirect.com/science/article/pii/S0021904514001324}{1--18}.


\bibitem{ban16}
E. Bannai, E. Bannai, K.T. Kim, W.-H. Yu and Y. Zhu, \emph{More on spherical designs of harmonic index t,}  
preprint \href{http://arxiv.org/abs/1507.05373}{arXiv:1507.05373}. 


\bibitem{buk16}
B. Bukh, \emph{Bounds on Equiangular Lines and on Related Spherical Codes}, SIAM J. Discrete Math., 30(1), 549–554.



\bibitem{del77}
P.~Delsarte, J.~M. Goethals, and J.~J. Seidel, \emph{Spherical codes and
  designs}, Geometriae Dedicata \textbf{6} (1977), 363--388.



\bibitem{grea14}
G. Greaves, J. H. Koolen, A. Munemasa, and F. Sz{\"o}ll{\H o}si,
\emph{Equiangular lines in {E}uclidean spaces}, J. Combin. Theory Ser. A, \textbf{138} (2016) 208-235.


\bibitem{Haa48}
J. Haantjes, \emph{Equilateral point-sets in elliptic two and three-dimensional spaces}, Nieuw Arch. Wisk., \textbf{22} (1948) 355-362.

\bibitem{kin16}
E. King and X. Tang, \emph{Computing upper bounds for equiangular lines in Euclidean spaces}, preprint available at \href{http://arxiv.org/abs/1606.03259}{arXiv:1606:03259}.

\bibitem{lin66}
J. H. van Lint and J.J. Seidel, \emph{Equiangular point sets in
elliptic geometry} Proc. Nedert. Akad. Wetensh. Series \textbf{69}
(1966), 335-348.


\bibitem{lem73}
P.~W.~H. Lemmens and J.~J. Seidel, \emph{Equiangular lines},
Journal of Algebra \textbf{24} (1973), \href{http://www.sciencedirect.com/science/article/pii/0021869373901233}{494--512}.

\bibitem{mus09}
O. Musin,
\emph{Spherical two-distance sets},
J. Combin. Theory Ser. A, Volume \textbf{116}, Issue 4, May 2009, 988-995.


\bibitem{neb12}
G. Nebe and B. Venkov, \emph{On tight spherical designs},
arXiv:1201.1830, (2012), 8pp.


\bibitem{oy16}
T. Okuda and W.-H. Yu,
\emph{A new relative bound for equiangular lines and nonexistence of tight spherical designs of harmonic index 4},
European J. of Combin. Volume \textbf{53}, April 2016,  \href{http://www.sciencedirect.com/science/article/pii/S0195669815002498}{96--103}.



\bibitem{yu16}
W.-H. Yu, \emph{There are no 76 equiangular lines in $\R^{19}$}, preprint available at \href{http://arxiv.org/abs/1511.08569}{arXiv:1511:08569}.

\end{thebibliography}

\end{document}